\documentclass[11pt, reqno,fleqn]{amsart}

\usepackage{amsmath}
\usepackage{amssymb}
\usepackage{amsfonts}
\usepackage{graphicx}
\usepackage{amsthm}
\usepackage{enumerate}
\usepackage{lscape}
\usepackage{dsfont}
\usepackage{color}

\newcommand{\R}{\mathds{R}}

\newcommand{\f}{\rightarrow}                  
\newcommand{\C}{\mathds{C}}            
\newcommand{\de}{\partial}

\newtheorem{theor}{Theorem}

\newtheorem{lem}[theor]{Lemma}

\begin{document}

\title{Global symplectic coordinates on gradient K\"ahler--Ricci solitons}
\author[A. Loi, M. Zedda]{Andrea Loi, Michela Zedda}
\address{Dipartimento di Matematica e Informatica, Universit\`{a} di Cagliari,
Via Ospedale 72, 09124 Cagliari, Italy}
\email{loi@unica.it; michela.zedda@gmail.com  }
\thanks{
The first author was (partially) supported by ESF within the program ÒContact and Symplectic TopologyÓ;
the second author was  supported by  RAS
through a grant financed with the ``Sardinia PO FSE 2007-2013'' funds and 
provided according to the L.R. $7/2007$.}
\date{}

\begin{abstract}
A  classical  result of D. McDuff \cite{mc} asserts that a simply--connected complete K\"ahler manifold $(M,g,\omega)$ with non positive sectional curvature admits global symplectic coordinates through a symplectomorphism $\Psi\!: M\f\R^{2n}$ (where $n$ is the complex dimension of $M$), satisfying the following property (proved by E. Ciriza in \cite{cr2}): the image
$\Psi (T)$ of any complex totally geodesic submanifold $T\subset M$ through the point $p$ such that $\Psi(p)=0$, is a complex linear subspace of $\C^n\simeq\R^{2n}$. The aim of this paper is to exhibit, for all positive integers $n$, examples of $n$-dimensional  complete K\"ahler manifolds with non-negative sectional curvature globally symplectomorphic to $\R^{2n}$ through a symplectomorphism satisfying Ciriza's property. 
\end{abstract}

\maketitle

\section{Introduction}
D. McDuff \cite{mc} (see also \cite{alcu}) proved a global version of  Darboux theorem
for  $n$-dimensional complete and simply--connected K\"ahler manifolds
with  nonpositive sectional curvature. She  shows that there exists a diffeomorphism $\Psi
:M\rightarrow \R^{2n}=\C^n$ satisfying
and $\Psi^*(\omega_{0})=\omega$, where
$\omega_0=\sum_{j=1}^ndx_j\wedge dy_j$ is the standard symplectic
form on $\R^{2n}$. The interest for these kind of questions
comes, for example, after Gromov's discovery \cite{gr} of the
existence of exotic symplectic structures on $\R^{2n}$.
E.  Ciriza \cite{cr2} (see also \cite{cr1} and \cite{cr3})
proves that the image $\Psi(T)$ of any complete complex and
totally geodesic  submanifold $T$ of  $M$ passing through the point $p$ such that $\Psi(p)=0$, is a
complex linear subspace of $\C^n$. A global symplectomorphism satisfying this property has been constructed by the first author and A. J. Di Scala \cite{DiScalaLoi08} (see also \cite{DiScalaLoiRoos}) for Hermitian symmetric spaces of noncompact type and by the authors of the present paper for the Calabi's inhomogeneous K\"ahler--Einstein metric on tubular domains (cfr. \cite{Calihom}). 
 It is then  natural and interesting   to investigate the existence of positively curved complete K\"ahler manifolds globally symplectomorphic to $\R^{2n}$ through a symplectomorphic satisfying the above Ciriza's property.
In this paper we construct explicit global symplectic coordinates for the positively curved  complete  gradient K\"ahler--Ricci solitons built by H. D. Cao in \cite{Cao}.  Moreover, we exhibit, for all positive integres $n$, an example of   gradient 
K\"ahler--Ricci solitons (the product of $n$ copies of the Cigar soliton)  where Ciriza's property holds true. Our results are summarized in the following two theorems (see next section for details and terminology).
\begin{theor}\label{soliton}
A gradient K\"ahler--Ricci soliton $(\C^n,\omega_{RS})$ is globally simplectomorphic to $(\R^{2n},\omega_0)$.
\end{theor}
\begin{theor}\label{cigar}
Let  $(\C^n,\omega_{C,n})$ be the the product of $n$ copies of the Cigar soliton.
Then there exists a simplectomorphism $\Psi_{C,n}:(\C^n,\omega_{C,n})\rightarrow (\R^{2n},\omega_0)$, with $\Psi_{C,n}(0)=0$, taking complete complex totally geodesic submanifolds through the origin  to complex linear subspaces of $\C^n\simeq\R^{2n}$.
\end{theor}

The paper  consists of two other sections containing respectively the basic material on gradient 
K\"ahler--Ricci solitons and the proofs of the main results.


\section{Gradient  K\"ahler--Ricci solitons}
We recall here what we need about the gradient K\"ahler--Ricci solitons described by H-D. Cao in \cite{Cao} (to whom we refer for references and further details). Let $g_{RS}$ be the K\"ahler metric on $\C^n$ generated by the radial K\"ahler potential $\Phi(z,\bar z)=u(t)$, where for all $t\in(-\infty,+\infty)$, $u$ is a smooth function of $t=\log(||z||^2)$ and as $t\f -\infty$ it has an expansion: 
\begin{equation}\label{zerocond}
u(t)=a_0+a_1e^t+a_2e^{2t}+\dots, \qquad a_1=1.
\end{equation}
Denote by $\omega_{RS}=\frac i2 \de\bar\de\Phi$ the K\"ahler form associated to $g_{RS}$. If $u$ satisfies the equation:
\begin{equation}
(u')^{n-1}u''e^{u'}=e^{nt},\nonumber
\end{equation}
then the conditions:
\begin{equation}\label{greaterthanzero}
u'(t)>0,\qquad u''(t)>0,\qquad\forall\, t\in(-\infty,+\infty),
\end{equation}
\begin{equation}\label{limitcond}
\lim_{t\f+\infty}\frac{u'(t)}{t}=n,\qquad \lim_{t\f+\infty}u''(t)=n.
\end{equation}
are fullfilled and  $(\C^n,\omega_{RS})$ is a gradient K\"ahler--Ricci soliton. The metric $g_{RS}$ is complete and positively curved and for $n=1$ one recovers the Cigar metric on $\C$ whose associated K\"ahler form reads:
$$\omega_C=\frac{dz\wedge d\bar z}{1+|z|^2},$$
which was introduced by Hamilton in \cite{Hamilton} as first example of K\"ahler--Ricci soliton on non-compact manifolds. Observe that a K\"ahler potential for $\omega_C$ is given by (see also \cite{Suzuki}):
\begin{equation}
\Phi_C=\int_0^{|z|}\frac{\log(1+s^2)}{s}ds.\nonumber
\end{equation}
Furthermore, in this case the Riemannian curvature reads:
\begin{equation}\label{Rcomp}
R=\frac{1}{(1+|z|^2)^{3}}.
\end{equation}
It is interesting observing that the K\"ahler metric $\omega_{C,n}$ on $\C^n=\frac i2\de\bar\de\Phi_{C,n}$ defined as product of $n$ copies of Cigar metric $\omega_C$, satisfies $\Phi_{C,n}=\Phi_C\oplus\dots\oplus\Phi_{C}$ and it is still a complete and positively curved (i.e. with non-negative sectional curvature) gradient K\"ahler--Ricci soliton, namely it satisfies (\ref{zerocond}), (\ref{greaterthanzero}) and (\ref{limitcond}) above.  
In particular its Riemannian tensor satisfies $R_{i\bar j k\bar l}=0$ whenever one of the indexes is different from the others and by (\ref{Rcomp}) it is easy to see that the nonvanishing components are given by:
\begin{equation}\label{Rjjjj}
  R_{j\bar jj\bar j}=\frac{1}{(1+|z_j|^2)^3}.
 \end{equation}

\section{Proof of the main results}

In \cite{LoiZuddas08} the first author of the present paper, jointly with F. Zuddas, proved the following result on the existence of a symplectomorphism between a rotation invariant K\"ahler manifold of complex dimension $n$ and $(\R^{2n},\omega_0)$. For the readers convenience, we summarize here that result and the proof in the case when the manifold is $\C^n$. This will be the main ingredient in the proof of our  main results.
\begin{lem}\label{fabio}
Let $\omega_\Phi = \frac{i}{2}
\partial \bar{\partial} \Phi$ be a rotation invariant 
K\"ahler form
on $\C^n$
 i.e. the K\"ahler potential
 only depends on $|z_j|^2$, $j=1, \dots ,n$. \footnote {Notice that the rotation invariant condition on the potential $\Phi$ is more general then the radial one
which requires $\Phi$ depending only on  $|z_1|^2+\cdots +|z_n|^2$.} If
\begin{equation}\label{cond0}
\frac{\partial \Phi}{\partial |z_k|^2} \geq 0, \ \  k=1, \dots ,n.
\end{equation}
 then the map:
 $$\Psi: (M, \omega_{\Phi}) \rightarrow ({\C}^n, \omega_0),\quad z=(z_1,\dots,z_n)\mapsto(\psi_1(z)z_1,\dots, \psi_n(z)z_n),$$
 where $$\psi_j=\sqrt{\frac{\partial \Phi}{\partial |z_j|^2}},\quad  j=1,\dots, n,$$
is a  symplectic immersion. If in addition:
 \begin{equation}\label{gencondb}
 \lim_{z \rightarrow +\infty} \sum_{j=1}^n
\frac{\partial \Phi}{\partial |z_j|^2} |z_j|^2 = + \infty,
\end{equation}
then $\Psi$ is a global symplectomorphism.
\end{lem}   
\begin{proof}
Assume condition (\ref{cond0}) holds true. Let us prove first that $F^*\omega_0=\omega$. We have:
\begin{equation}
\begin{split}
\Psi^*\omega_0=&\frac{i}{2}\sum_{j=1}^n d\Psi_j\wedge d\bar \Psi_j\\
=&\sum_{j=1}^n \left(\frac{\de \Psi_j}{\de z_j}dz_j+\frac{\de \Psi_j}{\de\bar z_j}d\bar z_j\right)\wedge \left(\frac{\de \bar \Psi_j}{\de z_j}dz_j+\frac{\de \bar \Psi_j}{\de\bar z_j}d\bar z_j\right)\\
=&\sum_{j,k=1}^n\left( \left|\frac{\de \Psi_j}{\de z_j}\right|^2- \left|\frac{\de\bar  \Psi_j}{\de z_j}\right|^2\right)dz_j\wedge d\bar z_j
\end{split}\nonumber
\end{equation}
Since
$$ \frac{\de \Psi_j}{\de z_j}= \frac{\de \psi_j}{\de z_j}z_j+\psi_j,\qquad \frac{\de \Psi_j}{\de \bar z_j}= \frac{\de \psi_j}{\de \bar z_j} z_j,$$
and
$$ \frac{\de \psi_j}{\de z_j}=\frac12 \psi_j^{-1} \left(\frac{\de^2 \Phi}{\de |z_j|^4} \right)\bar z_j,$$
it follows:
\begin{equation}
\begin{split}
\Psi^*\omega_0=&\sum_{j=1}^n\left( \left|\frac{\de \psi_j}{\de z_j}z_j+\psi_j\right|^2-\left|\frac{\de \psi_j}{\de z_j}\right|^2|z_j|^2\right)dz_j\wedge d\bar z_j\\
=&\sum_{j=1}^n\left( \frac{\de \psi_j}{\de z_j}\psi_jz_j+\frac{\de \psi_j}{\de \bar z_j}\psi_j\bar z_j+\psi_j^2\right)dz_j\wedge d\bar z_j\\
=&\sum_{j=1}^n\left(\left(\frac{\de^2 \Phi}{\de |z_j|^4} \right)|z_j|^2 +\left(\frac{\de \Phi}{\de |z_j|^2} \right)\right)dz_j\wedge d\bar z_j\\
=&\sum_{j=1}^n \frac{\de^2 \Phi}{\de z_j\de\bar z_j}dz_j\wedge d\bar z_j.
\end{split}\nonumber
\end{equation}
Observe now that since $\omega$ and $\omega_0$ are non-degenerate, it follows by the inverse function theorem that $\Psi$ is a local diffeomorphism.
If in addition condition (\ref{gencondb}) holds true, then  $\Psi$ is a proper map and hence a global diffeomorphism.
\end{proof}
We are now in the position of proving Theorem \ref{soliton}.
\begin{proof}[Proof of Theorem \ref{soliton}]
Let $\Phi (z, \bar z)= u(t)$, where $u(t)$ is given by (\ref{zerocond}). Then for all $j=1,\dots, n$
\begin{equation}
\frac{\de\Phi}{\de |z_j|^2}=\frac{\de\Phi}{\de ||z||^2}=\frac{u'(\log(||z||^2)}{||z||^2},\nonumber
\end{equation}
which is greater than zero for all $||z||^2\neq 0$ by (\ref{greaterthanzero}), and evaluated at $||z||^2=0$ gives  the value $1$ by (\ref{zerocond}).
Notice now that by  the first of the limit conditions given in (\ref{limitcond}) it follows that condition (\ref{gencondb}) in Lemma \ref{fabio} holds true.
Therefore by Lemma \ref{fabio} the map:
\begin{equation}
\!\!\!\!\!\!\!\!\!\!\!\!\!\!\!\!\!F\!:(\C^n,g_{RS})\f(\R^{2n},g_0),\quad z=(z_1,\dots,z_n)\mapsto\sqrt{\frac{u'(\log(||z||^2)}{||z||^2}}(z_1,\dots,z_n),\nonumber
\end{equation}
is the desired   global  symplectomorphism.
\end{proof}

In  order to prove  Theorem \ref{cigar} we need the following lemma  which classifies all totally geodesic submanifolds of $(\C^n,\omega_{C,n})$ through the origin.
\begin{lem}\label{totsub}
 Let $S$ be a totally geodesic complex submanifold (of complex dimension $k$) of $(\C^n, \omega_{C,n})$. Then, up to unitary transformation of $\C^n$, $S=(\C^k, \omega_{C,k})$.
\end{lem}
\begin{proof}
Let us first prove the statement for $n=2$. For $k=0,2$ there is nothing to prove, thus fix $k=1$. 
Let 
$$f\!:(S,\tilde \omega)\hookrightarrow(\C^2,\omega_{C,2}),\quad f(z)=(f_1(z),f_2(z)).$$
be a totally geodesic embedding of a $1$-dimensional complex manifold 
$(S,\tilde \omega)$ into  $(\C^2,\omega_{C,2})$.
 By  $\tilde \omega=f^*(\omega_{C,2})$ we get:
\begin{equation}\label{hprod}
\tilde \omega=\frac i2\left(\left|\frac{\de f_1}{\de z}\right|^2 \frac{1}{1+|f_1(z)|^2} +\left|\frac{\de f_2}{\de z}\right|^2 \frac{1}{1+|f_2(z)|^2}\right)dz\wedge d\bar z.
\end{equation}
Let $\tilde R$, $ R_C$ be the curvature tensor of $(S,\tilde \omega)$ and $(\C^2,\omega_C)$ respectively. Since $(S,\tilde \omega)$ is totally geodesic in $(\C^2, \omega_C)$ we have
$$\tilde R(X,JX,X,JX)= R_C(X,JX,X,JX)$$
for all the vector fields $X$ on $S$ (see e.g. \cite[p. 176]{KobayashiNomizu2}).
Taking $X=\de/\de z$, we have:
\begin{equation}\label{computeR}
\tilde R\left(\frac{\de}{\de z},\frac{\de}{\de \bar z},\frac{\de}{\de z},\frac{\de}{\de \bar z}\right)=-\frac{\de^2\tilde g}{\de z\de\bar z}+\tilde g^{-1}(z)\left|\frac{\de \tilde g(z)}{\de z}\right|^2,\nonumber
\end{equation}
where $\tilde g$ is the K\"ahler metric associated to $\tilde \omega$, i.e. 
$$\tilde g=\left|\frac{\de f_1}{\de z}\right|^2 \frac{1}{1+|f_1(z)|^2} +\left|\frac{\de f_2}{\de z}\right|^2 \frac{1}{1+|f_2(z)|^2}.$$
Further, since the vector field $\frac\de{\de z}$ corresponds through $\mathrm{d}f$ to $\frac{\de f_1}{\de z}\frac{\de}{\de z_1}+\frac{\de f_2}{\de z}\frac{\de}{\de z_2}$, by (\ref{Rjjjj}) we get:
\begin{equation}
 R_C\left(\frac{\de}{\de z},\frac{\de}{\de \bar z},\frac{\de}{\de z},\frac{\de}{\de \bar z}\right)=\left|\frac{\de f_1}{\de z}\right|^4 \frac{1}{(1+|f_1(z)|^2)^3}+\left|\frac{\de f_2}{\de z}\right|^4 \frac{1}{(1+|f_2(z)|^2)^3}.\nonumber
\end{equation}
Since
\begin{equation}
\frac{\de \tilde g}{\de z}=\sum_{j=1}^2 \left( \frac{2}{1+|f_j(z)|^2}\overline{\frac{\de f_j}{\de z}} \frac{\de^2 f_j}{\de z^2}- \left|\frac{\de f_j}{\de z}\right|^2\frac{\bar f_j}{(1+|f_j|^2)^2}\frac{\de f_j}{\de z}\right) ,\nonumber
\end{equation}
\begin{equation}
\begin{split}
\frac{\de^2 \tilde g}{\de z\,\de\bar z}=&\sum_{j=1}^2\left[\left|\frac{\de f_j}{\de z}\right|^4\frac{2|f_j|^2}{(1+|f_j(z)|^2)^3} +\left|\frac{\de^2 f_j}{\de z^2}\right|^2\frac{1}{1+|f_j(z)|^2} +\right.\\
&\left.-\frac{1}{(1+|f_j|^2)^2}\left(\bar f_j\left(\frac{\de f_j}{\de z}\right)^2\overline{\frac{\de^2 f_j}{\de z^2}}+\left|\frac{\de f_j}{\de z}\right|^4+f_j\left(\overline{\frac{\de f_j}{\de z}}\right)^2\frac{\de^2 f_j}{\de z^2}\right)\right]
\end{split}\nonumber
\end{equation}
after a long but straightforward computation, we get that $\tilde R\left(\frac{\de}{\de z},\frac{\de}{\de \bar z},\frac{\de}{\de z},\frac{\de}{\de \bar z}\right)- R_C\left(\frac{\de}{\de z},\frac{\de}{\de \bar z},\frac{\de}{\de z},\frac{\de}{\de \bar z}\right)$ assumes the form: $$\frac{-|A(f_1,f_2)|^2}{\left(\left|\frac{\de f_1}{\de z}\right|^2(1+|f_2|^2)+\left|\frac{\de f_2}{\de z}\right|^2(1+|f_1|^2)\right)(1+|f_1|^2)^2(1+|f_2|^2)^2},$$
where
\begin{equation}
\begin{split}
A(f_1,f_2)=&\left(\frac{\de^2 f_2}{\de z^2}\frac{\de f_1}{\de z}-\frac{\de^2 f_1}{\de z^2}\frac{\de f_2}{\de z}\right)(1+|f_1|^2)(1+|f_2|^2)+\\
&+\left(\frac{\de f_1}{\de z}\right)^2\frac{\de f_2}{\de z}\bar f_1(1+|f_2|^2)-\left(\frac{\de f_2}{\de z}\right)^2\frac{\de f_1}{\de z}\bar f_2(1+|f_1|^2).
\end{split}\nonumber
\end{equation}
Thus, $\tilde R\left(\frac{\de}{\de z},\frac{\de}{\de \bar z},\frac{\de}{\de z},\frac{\de}{\de \bar z}\right)- R_C\left(\frac{\de}{\de z},\frac{\de}{\de \bar z},\frac{\de}{\de z},\frac{\de}{\de \bar z}\right)=0$ iff $A(f_1,f_2
)=0$, i.e. iff
\begin{equation}\label{rmenor}
\begin{split}
&\frac{\de f_1}{\de z}(1+|f_2|^2)\left(\frac{\de^2 f_2}{\de z^2}(1+|f_1|^2)+\frac{\de f_1}{\de z}\frac{\de f_2}{\de z}\bar f_1\right)=\\
=&\frac{\de f_2}{\de z}(1+|f_1|^2)\left(\frac{\de^2 f_1}{\de z^2}(1+|f_2|^2)+\frac{\de f_2}{\de z}\frac{\de f_1}{\de z}\bar f_2\right),
\end{split}
\end{equation}
which is verified whenever one between $f_1(z)$ and $f_2(z)$ is constant (and thus zero since we assume $f(0,0)=0$), or when $f_1(z)=f_2(z)$.
In order to prove that these are the only solutions, write (\ref{rmenor}) as
\begin{equation}
\frac{\de f_1}{\de z}(1+|f_2|^2)\frac{\de}{\de z}\left(\frac{\de f_2}{\de z}(1+|f_1|^2)\right)=
\frac{\de f_2}{\de z}(1+|f_1|^2)\frac{\de}{\de z}\left(\frac{\de f_1}{\de z}(1+|f_2|^2)\right).\nonumber
\end{equation}
Assuming $f_1$, $f_2$ not constant, it leads to the equation:
\begin{equation}
\left( \frac{\frac{\de f_1}{\de z}(1+|f_2|^2)}{\frac{\de f_2}{\de z}(1+|f_1|^2)}\right)' \left(\frac{\de f_2}{\de z}(1+|f_1|^2)\right)^2=0,\nonumber
\end{equation}
which implies that for some complex constant $\lambda\neq 0$,
\begin{equation}\label{consta}
 \frac{\de f_1}{\de z}(1+|f_2|^2) =\lambda\frac{\de f_2}{\de z}(1+|f_1|^2),
\end{equation}
that is:
\begin{equation}
 \frac{\de \log f_1}{\de z}\bar f_1 =\lambda\frac{\de \log f_2}{\de z}\bar f_2.\nonumber
\end{equation}
Comparing the antiholomorphic parts we get $\bar f_1=\alpha \bar f_2$, for some complex constant $\alpha$. Substituting in (\ref{consta}) we get:
\begin{equation}
 \alpha(1+|f_2|^2) =\lambda(1+|\alpha|^2|f_2|^2).\nonumber
\end{equation}
Since $f(0,0)=0$, from this last equality follows $\alpha=\lambda$ and thus immediately $|\alpha|^2=1$. We have been proven that a totally geodesic submanifold of $(\C^2,\omega_{C,2})$ is, up to unitary transformation of $\C^2$, $(\C,\omega_C)$ realized either via the map $z\mapsto (f_1,0)$ (or equivalently $z\mapsto (0,f_1)$) or via $z\mapsto (f_1(z),\alpha f_1(z))$, with $|\alpha|^2=1$.

Assume now $S$ to be a $k$-dimensional complete totally geodesic complex submanifold of $(\C^n,\omega_{C,n})$ and let $\pi_j$, $j=1,\dots, n$, be the projection into the $j$th $\C$-factor in $\C^n$, $\pi_{jk}$ $j$, $k=1,\dots, n$, the projection into the space $\C^2$ corresponding to the $j$th and $k$th $\C$-factors.  Since $\pi_j(S)$, $j=1,\dots, n$, is totally geodesic into $(\C, \omega_C)$, it is either a point or the whole $\C$. Thus, up to unitary transformation of the ambient space, we can assume $S$ to be of the form:
\begin{equation}\label{esse}
(z_1,\dots,z_k)\mapsto(0,\dots, 0, h_{11}(z_1),\dots, h_{1r}(z_1),\dots, h_{k1}(z_k),\dots,h_{ks}(z_k)).
\end{equation}
Since also the projections $\pi_{jk}(S)$ have to be totally geodesic into $(\C^2,\omega_{C,2})$, by what we have proven for $n=2$, we can reduce (\ref{esse}) into the form:
\begin{equation}
(z_1,\dots,z_k)\mapsto(0,\dots, 0, h_{1}(z_1),\dots, \alpha_r h_{1}(z_1),\dots, h_{k}(z_k),\dots,\alpha_sh_{k}(z_k)),\nonumber
\end{equation}
where $|\alpha_t|^2=1$ for all $t$ appearing above. Thus, either $S=(\C^k, \omega_{C,k})$ or $S$ is a $k$ dimensional diagonal, which with a suitable unitary transformation can be written again as $(\C^k, \omega_{C,k})$, and we are done.
\end{proof}

\begin{proof}[Proof of Theorem \ref{cigar}]
The existence of a global symplectomorphism $\Psi_{C,n}\!:(\C^n,\omega_{C,n})\f(\R^{2n},\omega_0)$ is guaranteed  again by Lemma \ref{fabio}.
In fact for all $j=1,\dots, n$
\begin{equation}
\begin{split}
\frac{\de}{\de |z_j|^2}\Phi_{C,n}=&2\frac{\de}{\de |z_j|^2}\sum_{j=1}^n\int_0^{|z_j|}\frac{\log(1+s^2)}{s}ds\\
=&\frac{1}{|z_j|}\frac{\mathrm{d}}{\mathrm{d}|z_j|}\int_0^{|z_j|}\frac{\log(1+s^2)}{s}ds=\frac{\log(1+|z_j|^2)}{|z_j|^2}>0.
\end{split}\nonumber
\end{equation}
Moreover, condition (\ref{gencondb}) in Lemma \ref{fabio} is fullfilled by:
\begin{equation}
\lim_{z\f+\infty} |z_j|^2\sum_{j=1}^n\frac{\de\Phi_{C,n}}{\de |z_j|^2}=\lim_{z\f+\infty}\sum_{j=1}^n\log(1+|z_j|^2)=+\infty.\nonumber
\end{equation}
Thus by Lemma \ref{fabio} the map:
\begin{equation}\label{symplectomorphism}
\Psi_{C,n}\!: (\C^n, \omega_{C,n})\f(\R^{2n},\omega_0),\ \  z=(z_1,\dots,z_n)\mapsto (\psi_1(z_1)z_1,\dots, \psi_n(z_n)z_n),\nonumber
\end{equation}
with $$\psi_j=\sqrt{\frac{\log(1+|z_j|^2)}{|z_j|^2}},$$
is a global symplectomorphism. 

In order to prove the second part of the theorem, let $S$ be a $k$ dimensional totally geodesic complex submanifold of $(\C^n,\omega_{C,n})$ through the origin, which by Lemma \ref{totsub} is given by $(\C^k,\omega_{C,k})$. The image
$\Psi_{C,n} (S)$ is of the form:
$$\left(\sqrt{\frac{\log(1+|z_1|^2)}{|z_1|^2}}z_1,\dots, \sqrt{\frac{\log(1+|z_k|^2)}{|z_k|^2}}z_k,0,\dots, 0\right)\simeq \C^k,$$
concluding the proof.
\end{proof}

\end{document}